\documentclass[12pt,a4paper]{amsart}
\usepackage{amsmath,amssymb,amsthm,booktabs,comment,longtable}
\usepackage{enumitem}
\usepackage{lscape}
\usepackage{graphicx}
\usepackage[all]{xy}
\usepackage{ccaption}
\usepackage{stmaryrd}
\usepackage{hyperref}
\usepackage{ascmac}
\usepackage{datetime2}
\theoremstyle{definition}
\newtheorem{theorem}{Theorem}[section]
\newtheorem{lemma}[theorem]{Lemma}
\newtheorem{corollary}[theorem]{Corollary}
\newtheorem{proposition}[theorem]{Proposition}
\newtheorem{definition}[theorem]{Definition}
\newtheorem{fact}[theorem]{Fact}
\newtheorem{example}[theorem]{Example}
\newtheorem{setting}[theorem]{Setting}

\newtheorem{remark}[theorem]{Remark}

\newtheorem*{claim}{Claim}
\newtheorem{notation}[theorem]{Notation}

\newtheorem{note}[theorem]{Note}

\newcommand{\id}{\operatorname{id}}

\newcommand{\End}{\mathop{\mathrm{End}}\nolimits}

\newcommand{\R}{\mathbb{R}}
\newcommand{\Z}{\mathbb{Z}}

\newcommand{\pairing}[2]{\left\langle #1, #2 \right\rangle}

\newcommand{\A}{{}^\forall}
\newcommand{\E}{{}^\exists}

\newcommand{\inter}{\operatorname{int}}

\newcommand{\dom}{\operatorname{dom}}

\newcommand{\aspan}{\operatorname{aspan}}

\newcounter{mycounter}

\author{Koichi Tojo}
\address{Department of Mathematical Sciences, Tokai University, 4-1-1 Kitakaname, Hiratsuka-shi, Kanagawa 259-1292, Japan}
\email{koichi.tojo@tokai.ac.jp}

\author{Taro Yoshino}
\address{Graduate School of Mathematical Sciences, The University of Tokyo\\ 3-8-1 Komaba, Meguro-ku, Tokyo 153-8914, Japan
}
\email{yoshino@ms.u-tokyo.ac.jp}

\begin{document}
\title[Equivalence on exponential families]{Equivalence on exponential families as Hessian manifolds and classification of exponential families of order 1 with constant Hessian sectional curvature}
%MSC2020

%\thanks{The first author is supported by
%JSPS KAKENHI Grant Number JP16K17594.
%and the second author is supported by
%?? }
%\date{\today}
%\DTMnow
\maketitle
%\today
\begin{abstract}
We introduce an equivalence relation on exponential families, and 
prove that two exponential families are equivalent in this sense if and only if the corresponding induced Hessian manifolds are isomorphic. 
Moreover, we classify exponential families of order 1 with constant Hessian sectional curvature. 
To this end, we show that for an exponential family of order 1, it has constant Hessian sectional curvature if and only if the natural exponential family equivalent to the given family has a quadratic variance function (NEF-QVF). 
The classification coincides with that of NEF-QVF by Morris (1982) essentially. 
\end{abstract}

\tableofcontents

\section{Introduction and main theorem}\label{sec:intro}

In the field of information geometry, exponential family is one of the central objects. 
In fact, an open exponential family can be naturally regarded as a Hessian manifold, or equivalently a dually flat manifold. 
Therefore, it is natural to consider whether two given open exponential family are isomorphic as Hessian manifolds. 

We introduce an equivalence relation on exponential families (Definition~\ref{def:equivalence} and Proposition~\ref{proposition:equivalence_on_exp_families}). 
Moreover, we prove that the equivalence relation corresponds to the isomorphism class of Hessian manifolds (Theorem~\ref{theorem:main}). 
The following theorem holds in both measurable category and continuous category (see Sections~\ref{section:def_exp_family}, \ref{section:equivalence_on_exp_family}, and \ref{section:properties_of_exp_family} for more details). 

Correspondingly, let $X_1,X_2$ be measurable spaces or locally compact Hausdorff spaces. 
\begin{theorem}\label{theorem:main}
Let $\mathcal{P}_1$, $\mathcal{P}_2$ be open exponential families on $X_1$, $X_2$, respectively. 
Then $\mathcal{P}_1$ is equivalent to $\mathcal{P}_2$ as an exponential family in the sense of Definition~\ref{def:equivalence} if and only if 
$\mathcal{P}_1$ is isomorphic to $\mathcal{P}_2$ as a Hessian manifold. 
\end{theorem}

The theorem above enables us to classify exponential families satisfying a certain geometric property. 
Hessian sectional curvature is one of the important invariants of Hessian manifolds and has been investigated (\cite{shima95}, \cite{fk13},\cite{inoguchi25}). 
Molitor \cite{molitor22} classified exponential families on a finite set of order 1 with constant Hessian sectional curvature. 
We classify them without the restriction ``on a finite set''. 
Here, the point is the following theorem which translate the condition of constant Hessian sectional curvature into the property of exponential family:  

\begin{theorem}\label{theorem:NEF_and_QVF}
%[constant Hessian sectional curvatureとNEF-QVFの同値性]
Let $\mathcal{P}$ be a regular natural exponential family of order one. 
Then, $\mathcal{P}$ has a constant Hessian sectional curvature if and only if $\mathcal{P}$ has a quadratic variance function (NEF-QVF). 
Moreover, the Hessian sectional curvature is the coefficient of the quadratic term of the variance function. 
\end{theorem}

The definitions of natural exponential family and NEF-QVF are given later in Definitions~\ref{def:natural_exp_family} and \ref{def:QVF}, respectively. 
We can classify exponential families with constant Hessian sectional curvature due to the classification of NEF-QVFs given by \cite{morris82}. 
Here, we note that Morris's equivalence relation on exponential families slightly different from ours (See Remark~\ref{remark:equivalence_in_morris1982}).

\begin{proposition}\label{proposition:classification_table}
Any regular exponential family of order 1 with constant Hessian sectional curvature is equivalent to one of the following six families in Table~\ref{tab:1}. 
%1次元constant Hessian sectional curvatureを持つregular exponential familyはup to equivalenceで次の6系列の指数型分布族に限られる．
Here, note that in the table below, families other than the Poisson family have fixed parameters. For the normal family, they are equivalent even if the parameter $\sigma$ differs; however, in all other cases, differing parameters mean they are not equivalent as exponential families.

\noindent
\begin{table}[h]
\caption{Exponential families of order 1 with constant Hessian sectional curvature}\label{tab:1}
\begin{tabular}{c|c|c|c}
& \begin{tabular}{c}Normal\\($\sigma^2\in \R_{>0}$: fixed)\end{tabular} & Poisson& \begin{tabular}{c} Gamma\\($a\in \R_{>0}$: fixed)\end{tabular}  \\
\hline
$X$& $\R$& $\Z_{\geq 0}$& $\R_{>0}$  \\
\hline
density & $\frac{1}{\sqrt{2\pi\sigma^2}}e^{-(x-m)^2/2\sigma^2}$& $\frac{\lambda^xe^{-\lambda}}{x!}$& $\frac{(bx)^a}{\Gamma(a)}\frac{e^{-bx}}{x}$ \\
\hline
$\theta$ &$\frac{m}{\sigma^2}$ ($m\in \R$)&$\log\lambda$ ($\lambda\in \R_{>0})$ & $-b (b\in \R_{>0})$\\
\hline
$\varphi(\theta)$& $\frac{m^2}{2\sigma^2}=\frac{\sigma^2\theta^2}{2}$ & $\lambda=e^{\theta}$& $-a\log b$ \\
\hline
$\eta$ & $\sigma^2\theta$ & $\lambda$& $\frac{a}{b}=-\frac{a}{\theta}$  \\
\hline
$V(\eta)$ & $\sigma^2$ & $\lambda=\eta$& $\frac{a}{\theta^2}=\frac{1}{a}\eta^2$  \\
\hline
HSC & $0$ &  $0$ &  $\frac{1}{a}$ \\
\hline 
\hline
& \begin{tabular}{c}Binomial\\($n$: fixed) \end{tabular}& \begin{tabular}{c}Negative Binomial\\($r\in \R_{>0}$: fixed)\end{tabular} & \begin{tabular}{c}NEF-GHS\\ ($r\in \R_{>0}$: fixed)\end{tabular}\\
\hline
$X$ & $\{0,\cdots, n\}$& $\Z_{\geq 0}$ & $\R$ \\
\hline
density & $\binom{n}{x}p^x(1-p)^{n-x}$ & $\frac{\Gamma(x+r)}{\Gamma(r)x!}p^x(1-p)^r$ & $(2\cos\theta)^re^{\theta x}\frac{b_r(x)}{4\pi}$\\
\hline
$\theta$&$\log \frac{p}{1-p}$ ($0<p<1$) &$\log p$ $(0<p<1)$ & $\theta\in (-\frac{\pi}{2}, \frac{\pi}{2})$\\
\hline
$\varphi(\theta)$& \begin{tabular}{c}$-n\log(1-p)$\\$=n\log(1+e^{\theta})$\end{tabular}& \begin{tabular}{c}$-r\log (1-p)$\\$=-r\log(1-e^{\theta})$\end{tabular}& $-r\log \cos \theta$ \\
\hline
$\eta$& $\frac{n}{1+e^{-\theta}}$& $\frac{r}{e^{-\theta}-1}$ & $r\tan\theta$\\
\hline
$V(\eta)$& $\frac{ne^\theta}{(1+e^\theta)^2}=-\frac{1}{n}\eta^2+\eta$ & $\frac{re^{-\theta}}{(e^{-\theta}-1)^2}=\frac{1}{r}\eta^2+\eta$ & $r(1+\tan^2\theta)=r+\eta^2/r$\\
\hline
HSC & $-1/n$ & $1/r$ & $1/r$
\end{tabular}
\raggedright
See Remark~\ref{remark:explanation_of_table} for the definition of $b_r$. 
\end{table}
\end{proposition}

\begin{remark}\label{remark:explanation_of_table}
\begin{enumerate}
\item HSC in Table~\ref{tab:1} means the Hessian sectional curvature. 
\item In statistics, it is common to assume $r\in \Z_{>0}$ for the negative binomial distribution. However, here we also consider the case where $r\in \R_{>0}$.
\item For $r>0$, a function $b_r\colon \R\to \R_{>0}$ is defined by
\[
b_{r}(x):=B\left( \frac{r+ix}{2}, \frac{r-ix}{2}\right)=\frac{\Gamma(\frac{r+ix}{2})\Gamma(\frac{r-ix}{2}) }{\Gamma(r)}=\frac{\left| \Gamma(\frac{r+ix}{2})\right|^2 }{\Gamma(r)}. 
\]
\end{enumerate}
\end{remark}

\section{Definition of exponential family}\label{section:def_exp_family}
In this section, we recall the definition of an exponential family both in a continuous category and in a measurable category, 

First, we consider a continuous category. 
Let $X$ be a locally compact Hausdorff space and $\mathcal{R}(X)$ the set of all Radon measures on $X$. 
\begin{definition}[Exponential family]\label{def:exponential_family}
A non-empty subset $\mathcal{P}\subset \mathcal{R}(X)$ consisting of probability measures on $X$ is called an \emph{exponential family} on $X$ if there exists a triple $(E, T, \mu)$ satisfying the following four conditions: 
\begin{enumerate}[label=$(\roman*)$]
\item\label{item:def_exp_fam_vector_space} $E$ is a finite dimensional real vector space, 
\item\label{item:def_exp_fam_conti_map} $T\colon X\to E$ is a continuous map, 
\item\label{item:def_exp_fam_Radon} $\mu\in \mathcal{R}(X)$, 
\item\label{item:def_exp_fam_exp_form} for any $\nu\in \mathcal{P}$, there exists $\theta\in E^\vee$ such that 
\begin{align}
  d\nu(x)&=\exp(\pairing{\theta}{T(x)}-\varphi(\theta))d\mu(x)\quad (x\in X), \label{eq:exponential}\\
  \text{where }\varphi(\theta)&:=\log \int_{x\in X}\exp(\pairing{\theta}{T(x)})d\mu(x). 
\end{align}
\end{enumerate}
Here, $E^\vee$ denotes the dual space of $E$ and $\pairing{\cdot}{\cdot}\colon E^\vee\times E\to \R$ denotes the natural pairing. 
We call $\varphi(\theta)$ the \emph{log normalizer}, $X$ the \emph{sample space} of $\mathcal{P}$, 
the triple $(E, T, \mu)$ a \emph{realization} of $\mathcal{P}$, and $\mu$ the \emph{base measure} of $\mathcal{P}$.
Moreover, if the dimension of $E$ is minimal among all the realizations of $\mathcal{P}$, the realization is said to be \emph{minimal} and the dimension of $E$ is called the \emph{order} of $\mathcal{P}$.
\end{definition}

The definition in the category of measurable spaces and measurable maps ({\cite[\S 5]{bn70}}) is also common.
In this case, let $X$ be a measurable space, $\mathcal{P}$ a non-empty subset of the set of all probability measures on $X$. 
The conditions \ref{item:def_exp_fam_conti_map}, \ref{item:def_exp_fam_Radon} is replaced with the following conditions respectively: 
(ii)' $T\colon X\to E$ is a measurable map,  
(iii)' $\mu$ is a $\sigma$-finite measure on $X$, 
Here we regard $E$ as the Borel space coming from the natural topology on $E$. 

\begin{remark}
In Definition~\ref{def:exponential_family}, $\mu \in \mathcal{R}(X)$ becomes $\sigma$-finite measure on $X$ automatically. 
In fact, 
take any $\nu\in \mathcal{P}$, then we have $\mu\ll \nu$. 
Since $\nu$ is a probability measure, it is $\sigma$-finite. 
Therefore $\mu$ is $\sigma$-finite, which follows from the fact that 
a measure which is absolutely continuous with respect to some $\sigma$-finite measure is $\sigma$-finite. 
\end{remark}

\begin{remark}
\begin{enumerate}
\item In \cite{bn70}, the triple $(E,T, \mu)$ is called a representation of $\mathcal{P}$. 
\item In our setting, we do not need the assumption that $\mu$ is $\sigma$-finite (see {\cite[Definition~4.13]{ty21a}}) to define Radon--Nikod\'ym derivative. 
  \item 
    We assume that $\mathcal{P}$ admits a realization $(E, T, \mu)$, rather than fixing a specific realization. 
\end{enumerate}
\end{remark}

\section{Equivalence relation on exponential families}\label{section:equivalence_on_exp_family}
In this section, we define an equivalence relation on exponential families.
The following definition is valid in both the measurable and continuous categories.
Corresponding to the category, let $X_1,X_2$ be measurable spaces or locally compact Hausdorff spaces. 

\begin{definition}\label{def:equivalence}
  Let $\mathcal{P}_1$ and $\mathcal{P}_2$ be exponential families on $X_1$, $X_2$, respectively. We say that $\mathcal{P}_1$ and $\mathcal{P}_2$ are equivalent, which is denoted by $\mathcal{P}_1\simeq \mathcal{P}_2$, if 
there exist minimal realizations $(E, T_1, \mu_1)$ of $\mathcal{P}_1$ and $(E, T_2, \mu_2)$ of $\mathcal{P}_2$ such that ${T_1}_*\mathcal{P}_1={T_2}_*\mathcal{P}_2$ as a set of probability measures on $E$. 
Here, ${T_1}_*\mathcal{P}_1:=\{{T_1}_*\nu \mid \nu\in \mathcal{P}_1\}$, where ${T_1}_*\nu$ denotes the pushforward measure of $\nu$ by $T_1$. 
\end{definition}

\begin{proposition}\label{proposition:equivalence_on_exp_families}
The relation $\simeq$ introduced in Definition~\ref{def:equivalence} defines an equivalence relation on exponential families.
\end{proposition}
The proof is deferred to the Appendix.

Moreover, the following theorem holds in the both category. 
\begin{theorem}[\S~\ref{sec:intro}, Theorem~\ref{theorem:main}]
Let $\mathcal{P}_1$, $\mathcal{P}_2$ be open exponential families on $X_1, X_2$, respectively. 
Then $\mathcal{P}_1$ is equivalent to $\mathcal{P}_2$ as an exponential family in the sense of Definition~\ref{def:equivalence} if and only if 
$\mathcal{P}_1$ is isomorphic to $\mathcal{P}_2$ as a Hessian manifold. 
\end{theorem}
See Note~\ref{note:open_exp_is_Hessian_manifold} about the fact that an open exponential family has a Hessian structure naturally. 
We give the proof of this theorem in Section~\ref{sec:proof_of_main_theorem}. 

A simple sufficient condition for $\mathcal{P}_1$ and $\mathcal{P}_2$ to be equivalent is given as follows.

\begin{note}\label{note:sufficient_condition_for_equivalence}
Let $\mathcal{P}_1$, $\mathcal{P}_2$ be exponential families on $X_1$, $X_2$. 
If there exists a homeomorphism (bimeasurable map) $f\colon X_1\to X_2$ such that $f_*\mathcal{P}_1=\mathcal{P}_2$, then $\mathcal{P}_1\simeq \mathcal{P}_2$ holds. 
In fact, for a minimal realization $(E_1, T_1, \mu_1)$ of $\mathcal{P}_1$, the triple $(E_1, T_1\circ f^{-1}, f_*\mu_1)$ is a minimal realization of $\mathcal{P}_2$, and $T_{1*} \mathcal{P}_1= (T_1\circ f^{-1})_*\mathcal{P}_2$ holds. 
\end{note}

\begin{example}\label{example:equivalent_families}
Using Note~\ref{note:sufficient_condition_for_equivalence}, we can see that the following pairs are examples of equivalent families.
  \begin{itemize}
    \item The family of gamma distributions on $\R_{>0}$ and the family of inverse gamma distributions on $\R_{>0}$ via $f\colon x\mapsto x^{-1}$. 
    \item The family of normal distributions on $\R$ and the family of log normal distributions on $\R_{>0}$ via $f\colon x\mapsto e^x$. 
    \item The family of gamma distributions on $\R_{>0}$ and the family of log gamma distributions on $\R$ via $f\colon x\mapsto \log x$. 
  \item The families of normal distributions with fixed variances $\sigma_1$, $\sigma_2$ on $\R$ via
  $f\colon x\mapsto \frac{\sigma_1}{\sigma_2}x$. 
  \end{itemize}
\end{example}

\section{Properties of exponential family}\label{section:properties_of_exp_family}
In this section, we review basic properties of exponential families. 
The definitions and facts below is valid in both the measurable and continuous categories. 
Here we describe them in the continuous category for simplicity. 
Let $X$ be a locally compact Hausdorff space and $\mathcal{P}$ an exponential family on $X$.

\begin{fact}[{\cite[Theorem~5.2]{bn70}}]\label{fact:characterization_of_minimality}
  The realization $(E, T, \mu)$ of $\mathcal{P}$ is minimal if and only if the following two conditions hold:
  \begin{enumerate}[label=(\roman*)]
    \item \label{item:T_affinely_span} The set of functions $\{1, T_1, \cdots, T_n\}$ is linearly independent in $C(X)$. 
    \item \label{item:parameter_affinly_span} The subset $a(\mathcal{P})$ of $E^\vee$ affinely spans $E^\vee$. 
  \end{enumerate}
  Here, $T_i=p_i\circ T$, where $\{p_1,\cdots, p_n\}$ is a basis of $E^\vee$, $C(X)$ denotes the set of all continuous functions on $X$, 
  and $a\colon \mathcal{P}\to E^\vee$ is the map defined as in Fact~\ref{fact:parameter_map} below. 
\end{fact}

\begin{fact}[{\cite[\S~8.1]{bn78}}]\label{fact:parameter_map}
Under the condition Fact~\ref{fact:characterization_of_minimality} \ref{item:T_affinely_span}, for $\nu \in \mathcal{P}$, the element $\theta \in E^\vee$ satisfying Definition~\ref{def:exponential_family}~(\ref{eq:exponential}) is determined uniquely. Therefore the following map is well-defined and injective:
  \[a:=a_{E,T,\mu}\colon \mathcal{P}\to E^\vee, \nu\mapsto \theta.  \]
We call the map $a_{E,T,\mu}$ the parameter map associated with $(E,T,\mu)$. 
In what follows, the subscripts may be omitted when they are clear from the context.
\end{fact}

\begin{definition}[{\cite[\S~8.1]{bn78}}]\label{def:full}
  Let  $(E,T, \mu)$ be a minimal realization of $\mathcal{P}$. 
\begin{enumerate}
\item  $\mathcal{P}$ is \emph{full} if  
  \[a(\mathcal{P})=\left\{\theta\in E^\vee \ \middle|\  \int_{x\in X}\exp(\pairing{\theta}{T(x)})d\mu(x)<\infty\right\}. \]
\item $\mathcal{P}$ is \emph{open} if $a(\mathcal{P})\subset E^\vee$ is open in $E^\vee$. 
\item $\mathcal{P}$ is \emph{regular} if $\mathcal{P}$ is full and open. 
\end{enumerate}
These properties do not depend on the choice of the minimal realization. 
\end{definition}

\begin{note}[See \cite{amari82}, \cite{shima07}, for example, for more details]\label{note:open_exp_is_Hessian_manifold}
Any open exponential family $\mathcal{P}$ has a natural Hessian structure. 
In fact, Let $(E, T, \mu)$ be a minimal realization of $\mathcal{P}$. 
Then open subset $\Theta:=a(\mathcal{P})\subset E^\vee$ of the vector space $E^\vee$ can be regarded as a Riemannian manifold with the Fisher metric $g_F$. 
Moreover, the restriction of the standard affine connection on $E^\vee$ to $\Theta$ gives an affine connection on $\Theta$. 
Then it is well-known that $\nabla^2 \varphi=g_F$ holds. 
Here, $\varphi\colon \Theta\to \R$ is the log normalizer $\varphi(\theta)=\log \int_{x\in X}\exp(\pairing{\theta}{T(x)})d\mu(x)$. 
Thus, via the Hessian structure on $\Theta$, $\mathcal{P}$ has a natural Hessian structure. This does not depend on the choice of the minimal realization. 
\end{note}

Next, we let us recall the definition of natural exponential family. 

\begin{definition}\label{def:natural_exp_family}
Let $E$ be a finite dimensional vector space. 
An exponential family $\mathcal{P}$ on $E$ is called a \emph{natural exponential family (NEF)} on $E$ if a triple $(E, \id, \mu)$ is a minimal realization of $\mathcal{P}$ for some $\mu \in \mathcal{R}(E)$. 
\end{definition}

\begin{note}
Let $\mathcal{P}$ be an exponential family on $X$ and 
$(E, T, \mu)$ its minimal realization. 
Then $T_* \mathcal{P}$ is a natural exponential family on $E$ with a minimal realization $(E,\id, T_*\mu)$
\end{note}

\begin{remark}
In \cite{bn70}, the notion of exponential families is defined in the category of measurable spaces and measurable maps. On the other hand, we describe exponential families in the category of locally compact Hausdorff spaces and continuous maps, as this provides the advantage of avoiding ``up to $\mu$-a.e.'' technicalities. Although the two definitions are distinct in general, they turn out to be equivalent for natural exponential families (see Proposition~\ref{prop:NEF_is_independent_of_the_category}).
\end{remark}

\begin{proposition}\label{prop:NEF_is_independent_of_the_category}
The notion of natural exponential families in the category of measurable spaces and measurable maps coincides with that in the category of locally compact Hausdorff spaces and continuous maps. That is, a natural exponential family in the former category can naturally be viewed as one in the latter.
\end{proposition}
\begin{proof}
Let $E$ be a finite dimensional real vector space. 
If $\mathcal{P}$ is a NEF on $E$ in the continuous category, 
it is trivially a NEF on $E$ in the measurable category. 
We show the converse implication. 
We regard $E$ as a Borel space coming from the natural topology. 
Let $\mathcal{P}$ be a NEF on $E$ and $(E, \id, \mu)$ its minimal realization. 
Then, the conditions Definition~\ref{def:exponential_family} \ref{item:def_exp_fam_vector_space}, \ref{item:def_exp_fam_conti_map}, \ref{item:def_exp_fam_exp_form} is satisfied. 
Therefore, it is enough to show $\mu\in \mathcal{R}(E)$. 
Take any $\nu\in \mathcal{P}$. 
Then there exists $\theta \in E^\vee$ such that 
$d\nu(x)=\exp(\pairing{\theta}{x}-\varphi(\theta))d\mu(x)$. 
Since $\nu$ is a finite Borel measure on the finite dimensional vector space $E$, it is a Radon measure on $E$. 
Hence, $d\mu(x)=\exp(-\pairing{\theta}{x}+\varphi(\theta))d\nu(x)$ is also a Radon measure on $E$. 
\end{proof}

\section{Proof of Theorem~\ref{theorem:main}}\label{sec:proof_of_main_theorem}
In this section, we give a proof of Theorem~\ref{theorem:main}. 
To prove Theorem~\ref{theorem:main}, it is enough to show Lemmas~\ref{lemma:Hessian_structure_is_preserved_by_sufficient_statistics} and \ref{lemma:Hessian_equivalence_of_NEF}. 
We first prove Theorem~\ref{theorem:main} assuming these lemmas. 
Afterwards, we prove Lemma~\ref{lemma:Hessian_structure_is_preserved_by_sufficient_statistics} using Lemma~\ref{lemma:pushforwarded_exp_family}, and Lemma~\ref{lemma:Hessian_equivalence_of_NEF} using Corollary~\ref{corollary:Hessian_isom_induce_affine_trans} and Proposition~\ref{prop:equiv_and_affine_map_general}.

As in the previous section, we work in the continuous category. 
Let $X$ be a locally compact Hausdorff space. 

\begin{lemma}\label{lemma:Hessian_structure_is_preserved_by_sufficient_statistics}
Let $\mathcal{P}$ be an open exponential family on $X$ and $(E, T,\mu)$ its minimal realization. 
Then $\mathcal{P}$ is isomorphic to the natural exponential family $T_*\mathcal{P}$ on $E$ as a Hessian manifold. 
\end{lemma}

\begin{lemma}\label{lemma:Hessian_equivalence_of_NEF}
Let $\mathcal{P}_1$, $\mathcal{P}_2$ open natural exponential families on $E_1, E_2$, respectively. 
Then, $\mathcal{P}_1$ is equivalent to $\mathcal{P}_2$ if and only if $\mathcal{P}_1$ is isomorphic to $\mathcal{P}_2$ as a Hessian manifold. 
\end{lemma}

\begin{proof}[Proof of Theorem~\ref{theorem:main} using Lemmas~\ref{lemma:Hessian_structure_is_preserved_by_sufficient_statistics} and \ref{lemma:Hessian_equivalence_of_NEF}]
Let $\mathcal{P}_1$, $\mathcal{P}_2$ be open exponential families on $X_1$, $X_2$, respectively and
$(E_i, T_i, \mu_i)$ minimal realization of $\mathcal{P}_i$. 
``if part'':  Assume that $\mathcal{P}_1\simeq \mathcal{P}_2$ as Hessian manifolds. 
From Lemma~\ref{lemma:Hessian_structure_is_preserved_by_sufficient_statistics}, ${T_1}_*\mathcal{P}_1\simeq {T_2}_*\mathcal{P}_2$ as Hessian manifolds. 
Hence from Lemma~\ref{lemma:Hessian_equivalence_of_NEF}, ${T_1}_*\mathcal{P}_1\simeq {T_2}_*\mathcal{P}_2$ as exponential families. 
Since ${T_i}_*\mathcal{P}_i\simeq \mathcal{P}_i$ ($i=1,2$) as exponential families, 
by transitivity, $\mathcal{P}_1\simeq  \mathcal{P}_2$ as exponential families. 

``only if part'': Assume that $\mathcal{P}_1\simeq \mathcal{P}_2$ as exponential families. 
Then, ${T_1}_*\mathcal{P}_1$, ${T_2}_*\mathcal{P}_2$ are natural exponential family on $E_1, E_2$, respectively, and ${T_1}_*\mathcal{P}_1\simeq {T_2}_*\mathcal{P}_2$ as exponential families. 
Therefore, from Lemma~\ref{lemma:Hessian_equivalence_of_NEF}, 
${T_1}_*\mathcal{P}_1\simeq {T_2}_*\mathcal{P}_2$ as Hessian manifolds. 
Thus, from Lemma~\ref{lemma:Hessian_structure_is_preserved_by_sufficient_statistics}, $\mathcal{P}_1\simeq \mathcal{P}_2$ as Hessian manifolds. 
\end{proof}

\begin{lemma}\label{lemma:pushforwarded_exp_family}
  Let $\mathcal{P}$ be an exponential family on $X$ and $(E, T, \mu)$ a minimal realization of $\mathcal{P}$. 
\begin{enumerate}
  \item \label{item:pushforwarded_family_is_NEF}$T_*\mathcal{P}$ is natural exponential family on $E$ with a minimal realization $(E, \id, T_* \mu)$. 
Moreover, $a_{E, T, \mu}(\mathcal{P})=a_{E,\id, T^*\mu} (T_*\mathcal{P})$ holds and their log normalizers coincide. 
\item    If $\mathcal{P}$ is open (resp. full), $T_*\mathcal{P}$ is also open (resp. full). 
    \item $\mathcal{P}\simeq T_*\mathcal{P}$ as exponential families. 
\end{enumerate}
\end{lemma}

\begin{proof}
  (i): It is enough to show that $(E, \id, T_*\mu)$ is a minimal realization of $T_*\mathcal{P}$. 
  Let $\nu\in T_*\mathcal{P}$. 
  Then there exists $\tilde{\nu}\in \mathcal{P}$ such that $T_*\tilde{\nu}=\nu$. 
  Since $(E, T, \mu)$ is a realization of $\mathcal{P}$, 
  we can take $\theta \in E^\vee$ satisfying 
  \[
    d\tilde{\nu}(x)=\exp(\pairing{\theta}{T(x)}-\varphi(\theta))d\mu(x), 
  \]
where $\varphi(\theta)$ is the log normalizer. 
By the definition of pushforward, we have 
\[
  d\nu(v)=d(T_*\tilde{\nu})(v)=\exp(\pairing{\theta}{v}-\varphi(\theta))d(T_*\mu)(v) \quad (v\in E)
\]
Therefore, $(T_*\mu, E, \id)$ is a realization of $T_*\mathcal{P}$, and the parameter spaces and the log normalizers coincide. 
The minimality of $(T_*\mu, E, \id)$ follows from Fact~\ref{fact:characterization_of_minimality} and the minimality of $(E, T, \mu)$. 

(ii): From (i), we have $a_{E, T, \mu}(\mathcal{P})=a_{E,\id, T^*\mu} (T_*\mathcal{P})$. Therefore, the properties of openness and fullness are inherited. 

(iii): Since $(E, T, \mu)$ is a minimal realization of $\mathcal{P}$, 
$(E, \id, T_*\mu)$ is a minimal realization of $T_*\mathcal{P}$, 
and $T_*\mathcal{P}=\id_*(T_*\mathcal{P})$ holds, 
$\mathcal{P}\simeq T_*\mathcal{P}$ holds. 
  \end{proof}

Next, we show Lemma~\ref{lemma:Hessian_structure_is_preserved_by_sufficient_statistics}. 

\begin{proof}[Proof of Lemma~\ref{lemma:Hessian_structure_is_preserved_by_sufficient_statistics}]
Let $\mathcal{P}$ be an open exponential family on $X$, and $(E,T,\mu)$ its minimal realization. 
By Lemma~\ref{lemma:pushforwarded_exp_family} (\ref{item:pushforwarded_family_is_NEF}), 
$T_*\mathcal{P}$ is also an open exponential family with a minimal realization $(E, \id, T_*\mu)$. 
Moreover, the parameter space and the log normalizer coincide with those of $\mathcal{P}$. Therefore, the Hessian structures also coincide. 
\end{proof}

\begin{notation}\label{notation:equivalence}
For a finite dimensional real vector space $V$, we put $\overline{V}:=\R\oplus V\oplus \R$. 
For finite dimensional real vector spaces $V, W$, we define a set $D(V,W)$ by 
\begin{align*}
D(V, W):=\{ \ell \colon \overline{V}\to \overline{W} \text{ linear }\mid \E A \in \End(V, W), \E b\in W, c\in V^\vee, \delta \in \R \\\text{ such that } \ell=\begin{pmatrix}1 & c & \delta \\ & A&b \\ & & 1 \end{pmatrix}\}
\end{align*}

For $\ell =\begin{pmatrix}1 & c & \delta \\ & A&b \\ & & 1 \end{pmatrix}\in D(V,W)$, 
we define maps $\alpha_\ell$, $\beta_\ell$ as follows:
\begin{align*}
\alpha_\ell\colon V\to W,\ v\mapsto Av+b, \\
\beta_\ell \colon W^\vee \to V^\vee,\ \theta \mapsto \theta A+ c, 
\end{align*}
and define a map %maps $\ell^\# \colon C_0(W)\to C_0(V)$, 
$\ell_\# \colon \mathcal{R}(V)\to \mathcal{R}(W)$ as follows: 
\begin{align*}
%(\ell^\# f)(v)&:=e^{\pairing{c}{v}+\delta}f(\alpha_\ell (v))\\
\ell_\# \mu&:= (\alpha_\ell)_* (e^{\pairing{c}{v}+\delta}\mu). 
\end{align*}
\end{notation}

\begin{setting}\label{setting:paper17}
Let $E_1, E_2$ be finite dimensional real vector space and  
$\mathcal{P}_i$ an open exponential family on $E_i$ with a minimal realization $(E_i, \id, \mu_i)$ ($i=1,2$). 
For $i=1,2$, we put 
\begin{align*}
\Theta_i:=\left\{ \theta_i \in E_i^\vee \ \middle|\  \int_{x\in E_i}\exp( \pairing{\theta_i}{x})d\mu_i(x_i)<\infty\right\}. 
\end{align*}
For $\theta_i \in \Theta_i$, we put 
\begin{align*}
d\nu_{i, \theta_i}(x_i)&:=\exp(\pairing{\theta_i}{x_i}-\varphi_i(\theta_i))d\mu_i(x_i)\quad (x_i\in E_i), \\
\varphi_i(\theta_i)&:=\log \int_{x_i\in E_i}\exp(\pairing{\theta_i}{x_i})d\mu_i(x_i)\quad ( \theta_i \in \Theta_i). 
\end{align*}
\end{setting}

\begin{lemma}\label{lemma:detail_info_from_isom_of_Hessian_mfd}
Assume that $\Theta_1\simeq \Theta_2$ as Hessian manifolds. 
Then there exists $\ell=\begin{pmatrix}1 & c & \delta \\ & A&b \\ & & 1 \end{pmatrix}\in D(E_2,E_1)$ such that 
\begin{enumerate}
\item $\mu_1=\ell_\#\mu_2$
\item $\Theta_2=\beta_\ell(\Theta_1)$
\item $\varphi_1(\theta_1)=\varphi_2(\beta_\ell(\theta_1))+\pairing{\theta_1}{b}+\delta$ \quad $(\theta_1 \in \Theta_1)$
\item $\A \theta_1 \in \Theta_1$, $\nu_{1,\theta_1}=(\alpha_\ell)_* \nu_{2, \beta_\ell(\theta_1)}$. 
\item $A$ is invertible. 
\end{enumerate}
See Notation~\ref{notation:equivalence}
for affine maps $\alpha_\ell\colon E_2\to E_1$, $\beta_\ell\colon E_1^\vee \to E_2^\vee$. 
\end{lemma}
To show the lemma above, we use the following facts:
\begin{fact}\label{fact:extension_of_affine_map}
Let $V, V'$ be finite dimensional real vector spaces, $U\subset V$, $U'\subset V'$ be open subsets. 
Then for any affine map $f\colon U\to U'$, 
there exists affine map $F\colon V\to V'$ satisfying $F|_U=f$. 
\end{fact}

\begin{fact}\label{fact:measure_equality_and_Laplace_trans_equality}
Let $V$ be finite dimensional real vector space and 
$\mu, \nu$ measures  on $V$. 
Assume that $\inter (\dom L_\mu)\neq \emptyset$ holds. 
Then we have $\mu=\nu\iff L_\mu=L_\nu$. 
\end{fact}

\begin{proof}[Proof of Lemma~\ref{lemma:detail_info_from_isom_of_Hessian_mfd} ]
By the assumption $\Theta_1\simeq \Theta_2$ as Hessian manifolds, we fix a Hessian isomorphism $f\colon \Theta_1 \to \Theta_2$. Then, Fact~\ref{fact:extension_of_affine_map} allows us to choose an affine map $\beta \colon E_1^\vee \to E_2^\vee$ satisfying $\beta|_{\Theta_1}=f$. Furthermore, we can write $\beta(x_1)=x_1 A+c$ for some invertible $A\in \End(E_2, E_1)$ and $c\in E_2^\vee$. 
Here, we note that $A$ is invertible. 
Let $g_1,g_2$ be the Fisher metrics on $\Theta_1$, $\Theta_2$, respectively. 
By the equality $g_1=f^* g_2$, we have 
\begin{align*}
\nabla^2 \varphi_1=f^*(\nabla^2 \varphi_2)=\nabla^2 (\varphi_2\circ f). 
\end{align*}
Hence, we can and do take $b\in E_1$ and $\delta\in \R$ satisfying 
\begin{align}
\varphi_1(\theta_1)
=\varphi_2\circ f(\theta_1)+\pairing{\theta_1}{b}+\delta \quad (\theta_1\in \Theta_1).\label{eq:relation_of_lognormalizers} 
\end{align}
From the above discussion, 
we get $\ell:=\begin{pmatrix}1 & c & \delta \\ & A&b \\ & & 1 \end{pmatrix}\in D(E_2,E_1)$. 
We verify that $\ell$ satisfies the conditions (i) to (v). 
The condition (v) is already verified by the discussion above. 

(i): 
By Fact~\ref{fact:measure_equality_and_Laplace_trans_equality}, it is enough to show the following
\begin{claim}
$L_{\mu_1}=L_{\ell_\#\mu_2}$. 
\end{claim}
Take any $\theta_1\in \Theta_1$. 
Then we have 
\begin{align*}
L_{\mu_1}(\theta_1)
&=e^{\varphi_1(\theta_1)}\\
&=e^{\varphi_2(\beta(\theta_1))+\pairing{\theta_1}{b}+\delta}\\
&=e^{\pairing{\theta_1}{b}+\delta}\int_{x_2\in E_2}e^{\pairing{\beta(\theta_1)}{x_2}}d\mu_2(x_2)\\
&=e^{\pairing{\theta_1}{b}+\delta}\int_{x_2\in E_2}e^{\pairing{\theta_1 A+c}{x_2}}d\mu_2(x_2)\\
&=\int_{x_2\in E_2}e^{\pairing{\theta_1}{Ax_2+b}+\pairing{c}{x_2}+\delta}d\mu_2(x_2)\\
&=\int_{x_2\in E_2}e^{\pairing{\theta_1}{\alpha_\ell(x_2)}}e^{\pairing{c}{x_2}+\delta }d\mu_2(x_2)\\
&=\int_{x_1\in E_1}e^{\pairing{\theta_1}{x_1}}d(\ell_\#\mu_2)(x_1)\\
&=L_{\ell_\#\mu_2}(\theta_1). 
\end{align*}
This proves our claim. 

(ii): This follows from $\beta_\ell|_{\Theta_1}=f$. 

(iii): This follows from (\ref{eq:relation_of_lognormalizers}). 

(iv): 
Take any $\theta_1 \in \Theta_1$. 
By (i), we note that $d((\alpha_\ell)_*\mu_2)(x_1)=e^{-\pairing{c}{\alpha_\ell^{-1}(x_1)}-\delta}d\mu_1(x_1)$ holds. 
We have 
\begin{align*}
d((\alpha_\ell)_* \nu_{2,\beta_\ell(\theta_1)})(x_1)
&=\exp(\pairing{\beta_\ell(\theta_1)}{\alpha_\ell^{-1}(x_1)}-\varphi_2(\beta_\ell(\theta_1)))d((\alpha_\ell)_* \mu_2)(x_1)\\
&=\exp(\pairing{\beta_\ell(\theta_1)}{\alpha_\ell^{-1}(x_1)}-\varphi_1(\theta_1)+\pairing{\theta_1}{b}+\delta)e^{-\pairing{c}{\alpha_\ell^{-1}(x_1)}-\delta}d\mu_1(x_1)\\
&=\exp(\pairing{\beta_\ell(\theta_1)-c}{\alpha_\ell^{-1}(x_1)}-\varphi_1(\theta_1)+\pairing{\theta_1}{b})d\mu_1(x_1)\\
&=\exp(\pairing{\theta_1 A}{A^{-1}x_1-A^{-1}b}-\varphi_1(\theta_1)+\pairing{\theta_1}{b})d\mu_1(x_1)\\
&=\exp(\pairing{\theta_1}{x_1}-\varphi_1(\theta_1))d\mu_1(x_1)\\
&=d \nu_{1,\theta_1}(x_1).
\end{align*}
Thus, (iv) turn out to be true. 
\end{proof}

\begin{corollary}\label{corollary:Hessian_isom_induce_affine_trans}
Under Setting~\ref{setting:paper17}, 
If $\Theta_1\simeq \Theta_2$ holds as Hessian manifolds, 
There exists an invertible affine map $\alpha\colon E_2\to E_1$ satisfying $\alpha_*\mathcal{P}_2=\mathcal{P}_1$. 
\end{corollary}

\begin{proof}
Assume that $\Theta_1\simeq \Theta_2$ as Hessian manifolds. 
We can and do take $\ell\in D(E_2,E_1)$ satisfying the conditions (i) to (v) in Lemma~\ref{lemma:detail_info_from_isom_of_Hessian_mfd}. 
It is enough to show that 
$(\alpha_\ell)_* \mathcal{P}_2=\mathcal{P}_1$. 
\begin{align*}
\mathcal{P}_1
&=\{\nu_{1,\theta_1}\}_{\theta_1\in \Theta_1}\\
&=\{(\alpha_\ell)_* \nu_{2, \beta_\ell(\theta_1)}\}_{\theta_1\in \Theta}\\
&=\{(\alpha_\ell)_* \nu_{2, \theta_2}\}_{\theta_2\in \Theta_2}\\
&=(\alpha_\ell)_*\mathcal{P}_2. 
\end{align*}
\end{proof}

\begin{proof}[Proof of Lemma~\ref{lemma:Hessian_equivalence_of_NEF}]
($\Leftarrow$) By Corollary~\ref{corollary:Hessian_isom_induce_affine_trans}, there exists an invertible affine map $\alpha\colon E_2\to E_1$ satisfying $\alpha_*\mathcal{P}_2=\mathcal{P}_1$. Therefore, by Proposition~\ref{prop:equiv_and_affine_map_general}, $\mathcal{P}_1\simeq \mathcal{P}_2$ as exponential families. 

($\Rightarrow$)
Take a minimal realization $(E_i, \id, \mu_i)$ of $\mathcal{P}_i$. 
By Proposition~\ref{prop:equiv_and_affine_map_general}, 
we can and do take an affine map $\alpha\colon E_2\to E_1$ satisfying $\alpha_*\mathcal{P}_2=\mathcal{P}_1$. 
Then $(E_2, \alpha^{-1}, \alpha_* \mu_2)$ is also a minimal realization of $\mathcal{P}_1$. 
Therefore, 
there exists $\ell \in D(E_2,E_1)$ satisfying (i), (ii), (iii) in Lemma~\ref{lemma:relation_of_realizations}. 
Especially, $\beta_\ell\colon \Theta_1\to \Theta_2$ is an invertible affine transformation satisfying (iii)(c), it gives an isomorphism as Hessian manifolds  
\end{proof}

\section{Proof of Theorem~\ref{theorem:NEF_and_QVF}}\label{section:proof_of_theorem_on_NEF_and_QVF}
Our goal of this section is to give a proof of Theorem~\ref{theorem:NEF_and_QVF}. 

Although the definition of variance function is generalized to high dimensional case, 
for simplicity, we define it only in one dimensional case. 

Let $\mathcal{P}$ be a regular natural exponential family of order 1 on $\R$ with a minimal realization $(\R, \id_{\R}, \mu)$, namely, $\mathcal{P}=\{d\nu_\theta(x):=\exp( \theta x -\varphi(\theta))d\mu(x)\}_{\theta \in \Theta}$, where $\Theta:=\{\theta \in \R \mid \int_{x\in \R}e^{\theta x}d\mu(x)<\infty \}$,  $\varphi(\theta):=\log\int_{x\in \R}e^{\theta x}d\mu(x)$. 
We call $\theta\in\Theta$ a natural parameter. 

\begin{definition}[Variance function \cite{morris82}]
We define a mean parameter $\eta\in \Omega :=\varphi'(\Theta)$ and the variance function $V\colon \Omega\to \R$ as follows, respectively: 
\begin{align*}
\eta(\theta)&:=\int_{x\in X}xd\nu_\theta=\varphi'(\theta),\\
V(\eta)&:=\int_{x\in X}(x-\eta(\theta))^2\,d\nu_\theta(x)=\varphi''(\theta). 
\end{align*}
Here, we note that $\varphi' \colon \Theta \to \Omega$ is bijective by $\varphi''>0$. 
\end{definition}

\begin{definition}[NEF-QVF \cite{morris82}]\label{def:QVF}
We say $\mathcal{P}$ has a quadratic variance function if 
$V(\eta)$ is a polynomial of degree 2 or less with respect to the mean parameter $\eta$. 
In this case, $\mathcal{P}$ is called NEF-QVF. 
\end{definition}

For the proof of Theorem~\ref{theorem:NEF_and_QVF}, we use Fact~\ref{fact:HSC} below, which describes the Hessian sectional curvature in terms of the Fisher metric. 

\begin{fact}[{\cite[Proposition~1]{molitor22}}]\label{fact:HSC}
Let $(M,g,\nabla)$ be a Hessian manifold of dimension 1 with Hessian sectional curvature $S$. 
Let $x\colon U\to \R$ be an affine coordinate system defined on some open set $U\subset M$, with potential $\varphi\colon U\to \R$. 
Then for any nonzero symmetric contravariant tensor $\xi_p$ of degree 2 at $p\in U$, the following equality holds:
\[ S(\xi_p)=\frac{1}{2g} \frac{\partial^2}{\partial x^2}\log g. \]
Here, $g(x)=g(\frac{\partial}{\partial x}, \frac{\partial}{\partial x})$. 
\end{fact}

\begin{proof}[Proof of Theorem~\ref{theorem:NEF_and_QVF}]
Let $V$ be the variance function of $\mathcal{P}$
and put $y=\frac{d}{d\theta}\varphi(\theta)$. 
Let $g(\theta)d\theta^2$ be the Fisher metric on $\Theta$. By the definition of variance function, 
$g=\frac{dy}{d\theta}=V(y)$ holds. 
\begin{claim}
$\frac{d^2}{dy^2}V(y)=\frac{1}{g}\frac{d^2}{d\theta^2}\log g$
\end{claim}
In fact, differentiating $V(y)=g$ with respect to $y$, we get 
\begin{align*}
(\text{LHS})&=\frac{d}{dy}V(y), \\
(\text{RHS})&=\frac{d\theta}{dy}\frac{d}{d\theta}g=\frac{1}{g}\frac{d}{d\theta}g=\frac{d}{d\theta}\log g. 
\end{align*}
Therefore, we have
\[ \frac{d}{dy}V(y)=\frac{d}{d\theta}\log g. \]
Differentiating again with respect to $y$, we have
\begin{align*}
(\text{LHS})&=\frac{d^2}{dy^2}V(y),\\
(\text{RHS})&=\frac{d\theta}{dy}\frac{d^2}{d\theta^2}\log g=\frac{1}{g}\frac{d^2}{d\theta^2}\log g. 
\end{align*}
This completes the proof of our claim. 
By this claim, 
\begin{align*}
\Theta \text{ is NEF-QVF}&\iff V\text{ is a polynomial in }y \text{ of degree at most }2\\
&\iff \frac{d^2}{dy^2}V(y)\text{ is constant}\\
&\iff \frac{1}{g}\frac{d^2}{d\theta^2}\log g\text{ is constant}\\
&\iff \text{Hessian sectional curvature is constant.}
\end{align*}
Moreover, by Fact~\ref{fact:HSC} and Claim above, 
The Hessian sectional curvature is the coefficient of quadratic term of $V(y)$. 
\end{proof}

\section{Proof of Proposition~\ref{proposition:classification_table}}
In this section, we give a proof of Proposition~\ref{proposition:classification_table}. 
To this end, we use Theorem~\ref{theorem:NEF_and_QVF}, the classification of NEF-QVF by \cite{morris82} and Corollary~\ref{corollary:determinant} below. 

\begin{corollary}\label{corollary:determinant}
The signature of determinant of the variance function of NEF-QVF is invariant under the equivalence in the sense of Definition~\ref{def:equivalence}.
\end{corollary}
This corollary immediately follows from Proposition~\ref{prop:equiv_and_affine_map_general} and the following: 

\begin{fact}[\cite{morris82}] \label{fact:variance_function_and_equivalence}
Let $T\colon \R\to \R$ be a invertible affine transformation defined by $T(x)= ax+b$ ($a\in \R\setminus\{0\}, b\in \R$). 
Suppose that $\mathcal{P}$ is a NEF-QVF on $\R$ and 
put $\overline{\mathcal{P}}=T_*\mathcal{P}$. 
Then $\overline{\mathcal{P}}$ is also a NEF-QVF on $\R$, and its variance function $\overline{V}$ is given by $\overline{V}(m)=a^2V(\frac{m-b}{a})$. 
Here $V$ is the variance function of $\mathcal{P}$
\end{fact}

\begin{proof}[Proof of Proposition~\ref{proposition:classification_table}]
By Lemma~\ref{lemma:Hessian_structure_is_preserved_by_sufficient_statistics}, our classification problem is reduced to the case of regular natural exponential families of order 1. 
Moreover, by Theorem~\ref{theorem:NEF_and_QVF}, 
it is reduced to that of NEF-QVFs of order 1. 
The classification of NEF-QVFs by \cite{morris82} results in 6 series of families. 
Here we note that the equivalence relation by \cite{morris82} is weaker than ours as in Remark~\ref{remark:equivalence_in_morris1982}. 
Each of the 6 series is closed under the four operations in Remark~\ref{remark:equivalence_in_morris1982} (See {\cite[The end of \S~4]{morris82}}).
Additionally, the families of normal distributions $N(m, \sigma_1^2)$, $N(m,\sigma_2^2)$ with fixed variance are equivalent by Example~\ref{example:equivalent_families} even if $\sigma_1\neq \sigma_2$. 
Hence it is enough to show that Gamma, Negative Binomial, NEF-GHS are not equivalent in the sense of Definition~\ref{def:equivalence} 
since constant Hessian sectional curvature is an invariance of Hessian manifold and is given as the coefficient of quadratic term of the variance function by Theorem~\ref{theorem:NEF_and_QVF}.  
From Table~\ref{tab:1}, each of the determinant of the variance function is given as 0, positive, negative, respectively. Therefore, by Corollary~\ref{corollary:determinant}, these three families are not equivalent. 
\end{proof}

\begin{remark}[See {\cite[Section~3]{morris82}} for more details]\label{remark:equivalence_in_morris1982}
In \cite{morris82}, two NEFs on $\R$ are defined to be equivalent if they can be transformed into each other via the following four operations:
\begin{enumerate}
\item pushforward by an affine transformation,
\item convolution,
\item (infinite) division, and
\item generation.
\end{enumerate}
Our definition of equivalence in Definition~\ref{def:equivalence} corresponds exactly to Morris's operation (i) (Proposition~\ref{prop:equiv_and_affine_map_general}). Therefore, NEFs that are equivalent under Definition~\ref{def:equivalence} can be transformed into each other in the sense of \cite{morris82}. Note that operation (iv) corresponds to considering only full NEFs.
\end{remark}

\appendix

\section{Proof of Proposition~\ref{proposition:equivalence_on_exp_families}}

Since only transitivity is nontrivial, we prove it here (Corollary~\ref{corollary:transitive_law}). Specifically, we first prove Proposition~\ref{prop:equiv_and_affine_map_general} using Lemmas~\ref{lemma:composition_of_affine_is_realization} and \ref{lemma:relation_of_realizations} below, and then use it to establish transitivity.

In this section, we consider the measurable category. 
Let $X$ be a measurable space. 
\begin{lemma}\label{lemma:composition_of_affine_is_realization}
Let $\mathcal{P}$ be an exponential family on $X$ and 
$(E, T, \mu)$ its realization. 
Let $\tilde{E}$ be a finite dimensional real vector space and 
$\alpha\colon E\to \tilde{E}$ an invertible affine map. 
Then $(\tilde{E}, \alpha\circ T, \mu)$ is also a realization of $\mathcal{P}$
\end{lemma}

\begin{proof}
It is clear that $(\tilde{E}, \alpha\circ T, \mu)$ satisfies the conditions (i), (ii), (iii) in Definition~\ref{def:exponential_family}. 
We verify the condition (iv). 

Take any $\nu\in \mathcal{P}$. 
Since $(E,T, \mu)$ is a realization of $\mathcal{P}$, 
There exists $\theta\in E^\vee$ satisfying 
\[
d\nu(x)=\exp(\pairing{\theta}{T(x)}-\varphi(\theta))d\mu(x). 
\]

We put $\xi:=\theta\circ \alpha^{-1}\in \tilde{E}^\vee$. 
Noting that $\pairing{\theta}{T(x)}=\pairing{\xi \alpha}{T(x)}=\pairing{\xi}{\alpha T(x)}$, we have
\begin{align*}
d\nu(x)&=\exp(\pairing{\theta}{T(x)}-\varphi(\theta))d\mu(x)\\
&=\exp(\pairing{\xi}{\alpha T(x)}-\varphi(\theta))d\mu(x)
\end{align*}
Therefore (iv) is satisfied. 
\end{proof}

\begin{setting}\label{setting:relation_of_realizations}
Let $\mathcal{P}$ be an exponential family on $X$ with a realization $(E_i, T_i, \mu_i)$ ($i=1,2$). 
We put $\Theta_i:=\{\theta_i \in E_i^\vee \mid \int_{x\in X}\exp(\pairing{\theta_i}{T_i(x)})d\mu_i(x)<\infty\}$. 
For $\theta_i\in \Theta_i$, we set 
$d\nu_{i,\theta_i}(x)=\exp(\pairing{\theta_i}{T_i(x)}-\varphi_i(\theta_i))d\mu_i(x)$, where $\varphi_i(\theta_i)=\int_{x\in X}\exp(\pairing{\theta_i}{T_i(x)})d\mu_i(x)$ is the log normalizer.  
Hereafter, we write a.e. $x \in X$ to indicate that an equality holds almost everywhere with respect to some $\nu \in \mathcal{P}$. This is well-defined since any two elements of $\mathcal{P}$ are mutually absolutely continuous.
\end{setting}

\begin{lemma}[Reformulation of {\cite[Lemma 8.1]{bn78}}]\label{lemma:relation_of_realizations}
We assume Setting~\ref{setting:relation_of_realizations}. 
Furthermore we assume that $(E_1, T_1, \mu_1)$ is minimal. 
Then, there exists $\ell=\begin{pmatrix}1 & c & \delta \\ & A & b \\ & & 1 \end{pmatrix} \in D(E_2, E_1)$ satisfying the following conditions:
\begin{enumerate}
\item $T_1(x)= (\alpha_\ell \circ T_2)(x)$ (a.e. $x\in X$), 
\item $d\mu_1(x)=\exp(\pairing{c}{T_2(x)}+\delta)d\mu_2(x)$, 
\item For any $\theta_1 \in \Theta_1$, 
\begin{enumerate}
\item $\beta_\ell(\theta_1) \in \Theta_2$, 
\item ${\nu_1}_{\theta_1}={\nu_2}_{\beta_\ell(\theta_1)}$, 
\item $\varphi_1(\theta_1)=\varphi_2(\beta_\ell(\theta_1))+\pairing{\theta_1}{b}+\delta$. 
\end{enumerate}
\end{enumerate}
Moreover, if $(E_2,T_2, \mu_2)$ is also minimal, then $A$ is invertible. 
\end{lemma}

\begin{proof}
Since $(E_1, T_1, \mu_1)$ is minimal, $\aspan \Theta_1=E_1^\vee$ holds. 
Therefore, we can and do take $\theta_{1,0},\cdots, \theta_{1,n}\in\Theta_1$ satisfying $\aspan\{\theta_{0,1}, \cdots, \theta_{1,n}\}=E_1^\vee$. 
Here, we take $\theta_{2,0},\cdots, \theta_{2,n}\in \Theta_2$ so that $\nu_{1,\theta_{1,i}}=\nu_{2,\theta_{2,i}}$ holds ($i=0,\cdots,n$). 
We define an affine map $\beta\colon E_1^\vee\to E_2^\vee$ and an affine function $\gamma\colon E_1^\vee\to \R$ by 
  \[
    \beta(\theta_{1,i})=\theta_{2,i},\ \gamma(\theta_{1,i})=\varphi_1(\theta_{1,i})-\varphi_2(\theta_{2,i}). 
  \]
Then there exists a linear map $A\colon E_2\to E_1$, $c\in E_2^\vee$, $b\in E_1$ and $\delta\in \R$ such that 
  \begin{align*}
    \beta(\theta_1)&=\theta_1 A+c \quad (\theta_1 \in E_1^\vee),\\
    \gamma (\theta_1)&=\pairing{\theta_1}{b}+\delta \quad (\theta_1 \in E_1^\vee). 
\end{align*}
We verify that $A, b, c, \delta$ satisfy the conditions in Lemma~\ref{lemma:relation_of_realizations}. 
We put $\eta(x):=\log \frac{d\mu_1}{d\mu_2}(x)$. 

\noindent
\textbf{Claim 1}. For any $i \in \{0,\cdots, n\}$, a.e. $x\in X$, 
$\eta(x)=\pairing{\beta(\theta_{1,i})}{T_2(x)}-\pairing{\theta_{1,i}}{T_1(x)}+\gamma(\theta_{1,i})$. 

We show Claim 1. Take any $i\in \{0,\cdots, n\}$. By $\nu_{1,\theta_{1,i}}=\nu_{2,\theta_{2,i}}$, 
\begin{align*}
  &\exp(\pairing{\theta_{1,i}}{T_1(x)}-\varphi(\theta_{1,i}))d\mu_1(x)=\exp(\pairing{\theta_{2,i}}{T_2(x)}-\varphi_2(\theta_{2,i}))d\mu_2(x)\\
  \iff& \eta(x)=\pairing{\theta_{2,i}}{T_2(x)}-\pairing{\theta_{1,i}}{T_1(x)}+\varphi(\theta_{1,i})-\varphi_2(\theta_{2,i})\\
    \iff&\eta(x)=\pairing{\beta(\theta_{1,i})}{T_2(x)}-\pairing{\theta_{1,i}}{T_1(x)}+\gamma(\theta_{1,i}). 
  \end{align*}
This completes the proof of Claim 1.

  We verify the condition (ii). It is enough to show the following: \\
  \textbf{Claim 2}. $\eta(x)=\pairing{c}{T_2(x)}+\delta$ a.e. $x\in X$. \\
We show Claim 2. Take  $a_0,\cdots, a_n\in \R$ such that $\sum_i a_i=1$, $\sum_i a_i\theta_{1,i}=0$. 
  Multiplying both sides of Claim 1 by $a_i$ and summing over $i$, we obtain
  \begin{align*}
\eta(x)=\pairing{\beta(0)}{T_2(x)}-\pairing{0}{T_1(x)}+\gamma(0)=\pairing{c}{T_2(x)}+\delta \quad (\text{a.e. } x \in X)
  \end{align*}
Here we used the property that $\beta$, $\gamma$ are affine, 
This proves Claim 2. 

We verify the condition (i). 
By Claim 1, 2, for any $i\in \{0,\cdots, n\}$, the following equality holds:
\[ \pairing{\beta(\theta_{1,i})}{T_2(x)}-\pairing{\theta_{1,i}}{T_1(x)}+\gamma(\theta_{1,i})=\pairing{c}{T_2(x)}+\delta \quad (\text{a.e. }x\in X). \]
Since $\beta(\theta_{1,i})=\theta_{1,i} A+c, \gamma(\theta_{1,i})=\pairing{\theta_{1,i}}{b}+\delta$, the equality above is equivalent to the following: 
\begin{align*}
  \pairing{\theta_{1,i}}{AT_2(x)-T_1(x)+b}=0\quad (\text{a.e. }x\in X). 
\end{align*}
Hence, the condition $\aspan\{\theta_{1,0},\cdots, \theta_{1,n}\}=V^\vee$ implies (i). 

Finally we verify the condition (iii). 
Take any $\theta_1 \in \Theta_1$ and put 
$\theta_2:=\theta_1 A+c=\beta(\theta_1)\in E_2^\vee$. 
Take $a_0,\cdots, a_n\in \R$ such that $\sum_i a_i=1$, $\sum_i a_i\theta_{1,i}=\theta_1$. 
Then we have $\theta_2=\sum_i a_i \beta(\theta_{1,i})$. 
Multiplying both sides of Claim 1 by $a_i$ and summing over $i$, we get 
\[
\eta(x)=
  \pairing{\beta(\theta_{1})}{T_2(x)}-\pairing{\theta_1}{T_1(x)}+\gamma(\theta_1). 
\]
Therefore we have
\[
  \exp(\pairing{\theta_1}{T_1(x)})d\mu_1(x)=e^{\gamma(\theta_1)}\exp(\pairing{\beta(\theta_1)}{T_2(x)})d\mu_2(x), 
\]
which implies (a) $\beta(\theta_1)\in \Theta_2$, (b) $\nu_{1,\theta_1}=\nu_{2,\beta(\theta_1)}$. The condition (c) follows from
\begin{align*}
  e^{\varphi_1(\theta_1)}&=\int_{x\in X}\exp(\pairing{\theta_1}{T_1(x)})d\mu_1(x)\\
  &=e^{\gamma(\theta_1)}\int_{x\in X}\exp(\pairing{\beta(\theta_1)}{T_2(x)})d\mu_2(x)\\
  &=e^{\gamma(\theta_1)}e^{\varphi_2(\beta(\theta_1))}. 
\end{align*}
\end{proof}
\begin{proposition}\label{prop:equiv_and_affine_map_general}
Let $\mathcal{P}_i$ be an exponential family with a minimal realization $(E_i,T_i, \mu_i)$ ($i=1,2$). 
Then $\mathcal{P}_1\simeq \mathcal{P}_2\iff \E \alpha\colon E_2\to E_1$ invertible affine map such that 
$(\alpha\circ T_2)_*\mathcal{P}_2={T_1}_*\mathcal{P}_1$. 
\end{proposition}

\begin{proof}[Proof of Proposition~\ref{prop:equiv_and_affine_map_general}]
($\Leftarrow $)
By Lemma~\ref{lemma:composition_of_affine_is_realization}, 
$(E_1,\alpha\circ T_2, \mu_2)$ is also a minimal realization of $\mathcal{P}_2$ and satisfies
$(\alpha\circ T_2)_*\mathcal{P}_2={T_1}_*\mathcal{P}_1$. Thus $\mathcal{P}_1\simeq \mathcal{P}_2$ holds by definition. 

($\Rightarrow$)
There exists minimal realization
$(W_i, S_i, \eta_i)$ of $\mathcal{P}_i$ such that 
${S_1}_*\mathcal{P}_1={S_2}_*\mathcal{P}_2$. 
By Lemma~\ref{lemma:relation_of_realizations}, 
We can and do take invertible affine map $\beta_i\colon E_i\to W_i$ satisfying $S_i=\beta_i \circ T_i$ ($i=1,2$). 
We put 
$\alpha:=\beta_1^{-1}\circ \beta_2\colon E_2\to E_1$. 
Taking the pushforward of both sides of ${S_1}_*\mathcal{P}_1={S_2}_*\mathcal{P}_2$ by $\beta_1^{-1}$ yields
\begin{align*}
{T_1}_*\mathcal{P}_1
={\beta_1^{-1}}_*{S_1}_* \mathcal{P}_1
={\beta_1^{-1}}_*{S_2}_*\mathcal{P}_2
=(\beta_1^{-1}\circ \beta_2\circ T_2)_*\mathcal{P}_2=(\alpha\circ T_2)_*\mathcal{P}_2. 
\end{align*}
This proves the desired implication. 
\end{proof}

\begin{corollary}\label{corollary:transitive_law}
Let $\mathcal{P}_i$ be an exponential family on $X_i$ ($i=1,2,3$). 
Assume that $\mathcal{P}_1\simeq \mathcal{P}_2$ and $\mathcal{P}_2\simeq \mathcal{P}_3$ holds. 
Then 
$\mathcal{P}_1\simeq \mathcal{P}_3$ holds. 
\end{corollary}

\begin{proof}
Let $(E_i, T_i, \mu_i)$ be an minimal realization of $\mathcal{P}_i$. 
By the assumption and Proposition~\ref{prop:equiv_and_affine_map_general}, 
We can and do take invertible affine maps $\alpha_1\colon E_2\to E_1$, $\alpha_2\colon E_3\to E_2$ satisfying 
$(\alpha_1\circ T_2)_*\mathcal{P}_2={T_1}_*\mathcal{P}_1$ and 
$(\alpha_2\circ T_3)_*\mathcal{P}_3={T_2}_*\mathcal{P}_2$. 
Combining these equalities with the covariance of the pushforward yields
$(\alpha_1\circ \alpha_2\circ T_3)_*\mathcal{P}_3={T_1}_*\mathcal{P}_1$. 
Since $\alpha_1\circ \alpha_2\colon E_3\to E_1$ is also an invertibel affine map, again by Proposition~\ref{prop:equiv_and_affine_map_general}, we obtain $\mathcal{P}_1\simeq \mathcal{P}_3$. 
\end{proof}

\end{document}